\documentclass[a4paper,english,fontsize=11pt,parskip=half,abstracton]{scrartcl}
\usepackage{babel}
\usepackage[utf8]{inputenc}
\usepackage[T1]{fontenc}
\usepackage[a4paper,left=20mm,right=20mm,top=30mm,bottom=30mm]{geometry}
\usepackage{amsmath}
\usepackage{amsthm}
\usepackage{amssymb}
\usepackage{enumerate}
\usepackage{aliascnt}
\usepackage[bookmarks=true,
            pdftitle={Survey on perfect isometries},
            pdfauthor={Benjamin Sambale},
            pdfkeywords={},
            pdfstartview={FitH}]{hyperref}

\newtheorem{Thm}{Theorem}[section]
\newaliascnt{Lem}{Thm}
\newtheorem{Lem}[Lem]{Lemma}
\aliascntresetthe{Lem}
\newaliascnt{Prop}{Thm}
\newtheorem{Prop}[Prop]{Proposition}
\aliascntresetthe{Prop}
\newaliascnt{Cor}{Thm}
\newtheorem{Cor}[Cor]{Corollary}
\aliascntresetthe{Cor}
\newaliascnt{Con}{Thm}
\newtheorem{Con}[Con]{Conjecture}
\aliascntresetthe{Con}
\theoremstyle{definition}
\newaliascnt{Def}{Thm}
\newtheorem{Def}[Def]{Definition}
\aliascntresetthe{Def}
\newaliascnt{Ex}{Thm}
\newtheorem{Ex}[Ex]{Example}
\aliascntresetthe{Ex}
\newaliascnt{Rmk}{Thm}

\aliascntresetthe{Rmk}

\allowdisplaybreaks[1]

\renewcommand{\phi}{\varphi}
\newcommand{\C}{\operatorname{C}}
\newcommand{\N}{\operatorname{N}}
\newcommand{\Z}{\operatorname{Z}}
\newcommand{\PI}{\operatorname{PI}}
\newcommand{\ZZ}{\mathbb{Z}}
\newcommand{\CC}{\mathbb{C}}
\newcommand{\QQ}{\mathbb{Q}}

\newcommand{\id}{\mathrm{id}}
\newcommand{\J}{\mathrm{J}}
\newcommand{\AGammaL}{\operatorname{A\Gamma L}}
\newcommand{\GammaL}{\operatorname{\Gamma L}}

\newcommand{\Aut}{\operatorname{Aut}}

\newcommand{\Sym}{\operatorname{Sym}}

\newcommand{\GL}{\operatorname{GL}}
\newcommand{\SL}{\operatorname{SL}}

\newcommand{\Irr}{\operatorname{Irr}}
\newcommand{\IBr}{\operatorname{IBr}}

\newcommand{\Bl}{\operatorname{Bl}}
\newcommand{\Gal}{\operatorname{Gal}}

\newcommand{\Ker}{\operatorname{Ker}}

\newcommand{\diag}{\operatorname{diag}}

\mathchardef\ordinarycolon\mathcode`\:  %defines a nice ":=" 
 \mathcode`\:=\string"8000
 \begingroup \catcode`\:=\active
   \gdef:{\mathrel{\mathop\ordinarycolon}}
 \endgroup

\title{Survey on perfect isometries}
\author{Benjamin Sambale\footnote{Institut für Algebra, Zahlentheorie und Diskrete Mathematik, Leibniz Universität Hannover, Welfengarten 1, 30167 Hannover, Germany,
\href{mailto:sambale@math.uni-hannover.de}{sambale@math.uni-hannover.de}}}
\date{\today}

\begin{document}
\frenchspacing
\maketitle
\begin{abstract}\noindent
This paper is an introduction and a survey to the concept of perfect isometries which was first introduced by Michel Broué in 1990. Our main aim is to provide proofs of numerous results scattered in the literature. On the other hand, we make some observations which did not appear anywhere before.
\end{abstract}

\tableofcontents

\section{Introduction}

In 1990, Michel Broué~\cite{Broue} introduced the concept of \emph{perfect isometries} to relate the character theories of $p$-blocks of finite groups. 
Although his main interest was in Brauer correspondent blocks with abelian defect group, perfect isometries arose in more general settings as well. In the present paper we survey various definitions and properties related to perfect isometries which are scattered in the literature. We give examples and proofs whenever possible. In some places we will extend the existent literature, for instance by giving a characterization of nilpotent perfectly isometric blocks (see \autoref{hertweck}).
On the other hand, we will only work on the level of characters and do not employ higher categorical concepts like derived equivalences.  

The paper is organized as follows. In the second section we revisit Broué's original definition of perfect isometries and deduce their basic properties.
After that, we consider the role of the signs coming from perfect isometries. 
In particular, we discuss the group of self perfect isometries to resolve a conjecture by Ruengrot~\cite{Ruengrot}. 
In the following section we investigate a number of invariants of blocks which are preserved by perfect isometries. This is useful for distinguishing perfect isometry classes. In Section~5 we prove that nilpotent blocks are perfectly isometric if and only if their defect groups have the same character table. One direction of this equivalence follows easily from Broué--Puig's theorem on nilpotent blocks~\cite{BrouePuig}. The other direction uses a result of Hertweck~\cite{Hertweck}.
In the next section we generalize a sufficient criterion for the existence of perfect isometries by Horimoto--Watanabe~\cite{WatanabePerfIso}. This naturally leads us to Broué's notion of \emph{isotypies} and Brauer's notion of the \emph{type} of a block. Finally in the last section, we give an overview of Broué's Conjecture on blocks with abelian defect groups. 

\section{Definitions and their justification}

Our notation is fairly standard and can be found in Navarro~\cite{Navarro}. 
For the convenience of the reader we recall the basics. 
The cyclotomic field of degree $n$ is denoted by $\QQ_n$. 
For a finite group $G$ and a prime $p$, we denote the set of $p$-regular elements of $G$ by $G^0$. The $p$-part of the order of $G$ is $|G|_p$. For $g\in G$ let $g_p$ and $g_{p'}$ be the $p$-factor respectively the $p'$-factor of $g$. Let $\Irr(G)$ and $\IBr(G)$ be the sets of irreducible complex characters and irreducible Brauer characters of $G$ respectively. The corresponding sets of generalized characters are denoted by $\ZZ\Irr(G)$ and $\ZZ\IBr(G)$. %The trivial character of $G$ is denoted by $1_G$.
For class functions $\chi$, $\psi$ of $G$ (or of $G^0$) let
\begin{align*}
[\chi,\psi]&:=\frac{1}{|G|}\sum_{g\in G}\chi(g)\psi(g^{-1}),\\
[\chi,\psi]^0&:=\frac{1}{|G|}\sum_{g\in G^0}\chi(g)\psi(g^{-1}).
\end{align*}
Every $\chi\in\Irr(G)$ gives rise to a primitive central idempotent
\[e_\chi:=\frac{\chi(1)}{|G|}\sum_{g\in G}\chi(g^{-1})g\in\Z(\CC G).\]
The decomposition numbers $d_{\chi\psi}$ are defined by $\chi^0:=\chi_{G^0}=\sum_{\phi\in\IBr(G)}d_{\chi\phi}\phi$. 
More generally, for a $p$-element $u\in G$ there exist generalized decomposition numbers $d^u_{\chi\psi}$ such that
\[\chi(us)=\sum_{\phi\in\IBr(\C_G(u))}d^u_{\chi\phi}\phi(s)\qquad(s\in\C_G(u)^0).\]
Every $\phi\in\IBr(G)$ determines a projective indecomposable character $\Phi_\phi:=\sum_{\chi\in\Irr(G)}d_{\chi\psi}\chi$. 

At the moment it suffices to consider blocks as sets of irreducible characters (as they are introduced in Isaacs~\cite[Definition~15.17]{Isaacs}). In Section~3 we will start working with $p$-modular systems.
In the following we fix $p$-blocks $B$ and $B'$ of finite groups $G$ and $H$ respectively. 

\begin{Def}[Broué~{\cite[Définition~1.1]{Broue}}]\label{broue}
An isometry $I:\ZZ\Irr(B)\to\ZZ\Irr(B')$ (with respect to the usual scalar product) is called \emph{perfect} if the map
\[\mu_I:G\times H\to\CC,\qquad (g,h)\mapsto\sum_{\chi\in\Irr(B)}\chi(g)I(\chi)(h)\]
satisfies
\begin{enumerate}
\item[(sep)] If exactly one of $g$ and $h$ is $p$-regular, then $\mu_I(g,h)=0$.
\item[(int)] $\mu_I(g,h)/\lvert\C_G(g)\rvert_p$ and $\mu_I(g,h)/\lvert\C_H(h)\rvert_p$ are algebraic integers for $g\in G$ and $h\in H$.
\end{enumerate}
In this case we say that $B$ and $B'$ are \emph{perfectly isometric}.
\end{Def}

Since $\Irr(B)$ is an orthonormal basis of $\ZZ\Irr(B)$, we have $I(\chi)\in\pm\Irr(B')$ for every $\chi\in\Irr(B)$ in the situation of \autoref{broue}. It follows that $\mu_I\in\ZZ\Irr(B\otimes B')$ where $B\otimes B'$ describes the block of $G\times H$ consisting of the characters $\Irr(B)\times\Irr(B')$. If $\mu_I$ is given, one can recover $I$ via
\begin{equation}\label{alter}
\begin{split}
I(\chi)(h)&=\sum_{\psi\in\Irr(B)}I(\psi)(h)[\chi,\psi]=\frac{1}{|G|}\sum_{\psi\in\Irr(B)}I(\psi)(h)\sum_{g\in G}\chi(g)\psi(g^{-1})\\
&=\frac{1}{|G|}\sum_{g\in G}\chi(g)\sum_{\psi\in\Irr(B)}\psi(g^{-1})I(\psi)(h)=\frac{1}{|G|}\sum_{g\in G}\chi(g)\mu_I(g^{-1},h)
\end{split}
\end{equation}
for $\chi\in\Irr(B)$ and $h\in H$. 

In the following we will often write a perfect isometry $I:\ZZ\Irr(B)\to\ZZ\Irr(B')$ in the form $I(\chi)=\epsilon_I(\chi)\widehat{I}(\chi)$ where $\widehat{I}:\Irr(B)\to\Irr(B')$ is a bijection and $\epsilon_I(\chi)=\pm 1$.

\begin{Prop}[Kiyota~{\cite[Theorem~2.2]{KiyotaPI}}]
Condition \textnormal{(int)} can be replaced by
\begin{enumerate}
\item[\textnormal{(int')}] If $g$ and $h$ are $p$-singular, then $\mu_I(g,h)/\lvert\C_G(g)\rvert_p$ and $\mu_I(g,h)/\lvert\C_H(h)\rvert_p$ are algebraic integers.
\end{enumerate}
\end{Prop}
\begin{proof}
Assuming that (sep) and (int') hold, it suffices to show (int) for $p$-regular elements $g$ and $h$. The class function $\psi_g:H\to\CC$, $x\mapsto\mu_I(g,x)$ vanishes on the $p$-singular elements by (sep). Hence by \cite[Theorem~2.13]{Navarro} there exist $a_\phi\in\CC$ for $\phi\in\IBr(H)$ such that $\psi_g=\sum_{\phi\in\IBr(H)}a_\phi\Phi_\phi$. Moreover, \[a_\phi=[\psi_g,\phi]^0=[\psi_g,\hat{\phi}]=\sum_{\chi\in\Irr(B)}\chi(g)[I(\chi),\hat\phi]\] 
are algebraic integers by \cite[Lemma~2.15]{Navarro}. Therefore, by \cite[Lemma~2.21]{Navarro}, also 
\[\frac{\mu_I(g,h)}{\lvert\C_H(h)\rvert_p}=\frac{\psi_g(h)}{\lvert\C_H(h)\rvert_p}=\sum_{\phi\in\IBr(H)}a_\phi\frac{\Phi_\phi(h)}{\lvert\C_H(h)\rvert_p}\]
is an algebraic integer. Similarly, $\mu_I(g,h)/\lvert\C_G(g)\rvert_p$ is an algebraic integer. 
\end{proof}

\begin{Prop}\label{equiv}
Perfect isometry is an equivalence relation.
\end{Prop}
\begin{proof}
We show first that the identity on $\ZZ\Irr(B)$ is perfect (this is also explicit in \cite[Lemma~3.1]{Sanus}). By \cite[Corollary~5.11 (block orthogonality)]{Navarro}, (sep) holds and for (int) we may assume that $g_p=h_p^{-1}=:x$. Then by \cite[Lemmas~5.1 and 5.13]{Navarro} we have
\begin{align*}
\mu_\id(g,h)&=\sum_{\chi\in\Irr(B)}\chi(g)\overline{\chi}(h^{-1})=\sum_{\chi\in\Irr(B)}\sum_{\phi,\mu\in\IBr(\C_G(x))}d_{\chi\phi}^x\overline{d_{\chi\mu}^x}\phi(g_{p'})\mu(h_{p'})\\
&=\sum_{\phi,\mu\in\IBr(\C_G(x))}\phi(g_{p'})\mu(h_{p'})\sum_{\chi\in\Irr(B)}d_{\chi\phi}^x\overline{d_{\chi\mu}^x}=\sum_{\phi,\mu\in\IBr(\C_G(x))}\phi(g_{p'})\mu(h_{p'})c_{\phi\mu}\\
&=\sum_{\phi\in\IBr(\C_G(x))}\phi(g_{p'})\Phi_\phi(h_{p'})=\sum_{\mu\in\IBr(\C_G(x))}\mu(h_{p'})\Phi_\mu(g_{p'}).
\end{align*}
Now the claim follows from \cite[Lemma~2.21]{Navarro}, since $\C_G(x)\cap\C_G(g_{p'})=\C_G(g)$ and $\C_G(x)\cap\C_G(h_{p'})=\C_G(h)$.

Next let $I:\ZZ\Irr(B)\to\ZZ\Irr(B')$ be a perfect isometry. Then $I^{-1}:\ZZ\Irr(B')\to\ZZ\Irr(B)$ is an isometry. For $g\in G$ and $h\in H$ we have
\begin{equation}\label{inverse}
\mu_I(g,h)=\sum_{\chi\in\Irr(B)}\chi(g)I(\chi)(h)=\sum_{\psi\in\Irr(B')}I^{-1}(\psi)(g)\psi(h)=\mu_{I^{-1}}(h,g).
\end{equation}
This shows that $I^{-1}$ is perfect.

Finally, let $I:\ZZ\Irr(B)\to\ZZ\Irr(B')$ and $J:\ZZ\Irr(B')\to\ZZ\Irr(B'')$ be perfect isometries where $B''$ is a block of a finite group $K$. We need to show that the isometry $JI=J\circ I$ is perfect. For $g\in G$ and $k\in K$ we have
\begin{align*}
\frac{1}{|H|}\sum_{h\in H}\mu_I(g,h^{-1})\mu_J(h,k)&=\frac{1}{|H|}\sum_{h\in H}\sum_{\chi\in\Irr(B)}\sum_{\psi\in\Irr(B')}\chi(g)I(\chi)(h^{-1})\psi(h)J(\psi)(k)\\
&=\frac{1}{|H|}\sum_{\chi\in\Irr(B)}\chi(g)\sum_{\psi\in\Irr(B')}J(\psi)(k)\sum_{h\in H}I(\chi)(h^{-1})\psi(h)\\
&=\sum_{\chi\in\Irr(B)}\chi(g)\sum_{\psi\in\Irr(B')}J(\psi)(k)[I(\chi),\psi]\\
&=\sum_{\chi\in\Irr(B)}\chi(g)JI(\chi)(k)=\mu_{JI}(g,k).
\end{align*}
If exactly one of $g$ and $k$ is $p$-regular, then $\mu_I(g,h^{-1})=0$ or $\mu_J(h,k)=0$ for every $h\in H$. Hence, $JI$ fulfills (sep). To prove (int), let $\mathcal{R}$ be a set of representatives for the conjugacy classes of $H$. Then
\begin{align*}
\frac{\mu_{JI}(g,k)}{\lvert\C_G(g)\rvert_p}&=\frac{1}{|H|\lvert\C_G(g)\rvert_p}\sum_{h\in\mathcal{R}}|H:\C_H(h)|\mu_I(g,h^{-1})\mu_J(h,k)\\
&=\sum_{h\in\mathcal{R}}\frac{\mu_I(g,h^{-1})}{\lvert\C_G(g)\rvert_p}\frac{\mu_J(h,k)}{\lvert\C_H(h)\rvert}
\end{align*}
and $|H|_{p'}\mu_{JI}(g,k)/\lvert\C_G(g)\rvert_p$ is an algebraic integer. Since $\lvert\C_G(g)\rvert_p$ and $|H|_{p'}$ are coprime and $\mu_{JI}(g,k)$ is an algebraic integer as well, it follows that $\mu_{JI}(g,k)/\lvert\C_G(g)\rvert_p$ is an algebraic integer.
The same holds for $\mu_{JI}(g,k)/\lvert\C_K(k)\rvert_p$ and we are done. 
\end{proof}

%\goodbreak
\begin{Ex}\label{ex1}\hfill
\begin{enumerate}[(i)]
\item\label{ex1a} Every $\alpha\in\Aut(G)$ induces a perfect isometry $\ZZ\Irr(B)\to\ZZ\Irr(\alpha(B))$ where it is understood that the action of $\Aut(G)$ on $\Irr(G)$ permutes blocks. 

\item Every $\gamma\in\Gal(\QQ_{|G|}|\QQ)$ induces a perfect isometry $\ZZ\Irr(B)\to\ZZ\Irr(\gamma(B))$ where again $\Gal(\QQ_{|G|}|\QQ)$ acts on $\Irr(G)$ (for a stronger claim see Kessar~\cite{Kessarisotypies}).

\item\label{lin} Let $\lambda\in\Irr(G)$ with $\lambda(1)=1$. Then the characters $\{\lambda\chi:\chi\in\Irr(B)\}\subseteq\Irr(G)$ form a block $\lambda B$ and the map $\Irr(B)\to\Irr(\lambda B)$, $\chi\mapsto\lambda\chi$ induces a perfect isometry. This shows that perfect isometries do not commute with (Galois) automorphisms (consider the cyclic group $G\cong C_3$ with $p=3$ for instance). Moreover, if $\Irr(B)$ contains a linear character, then $B$ is perfectly isometric to the principal block of $G$.

\item The natural epimorphism $G\to G/\Ker(B)$ induces a perfect isometry between $B$ and the dominated block $\overline{B}$ of $G/\Ker(B)$ (see \cite[p. 198]{Navarro} for a definition). 

\item Let $b$ be a block of $N\unlhd G$ with inertial group $T\le G$ and let $B$ be a block of $T$ covering $b$. Then the Fong-Reynolds correspondence $\Irr(B)\to\Irr(B^G)$, $\chi\mapsto\chi^G$ induces a perfect isometry. Indeed, for $g\in T$ and $h\in G$ we have
\[\sum_{\chi\in\Irr(B)}\chi(g)\chi^G(h)=\sum_{i=1}^s\frac{\lvert\C_G(h)\rvert}{\lvert\C_T(h_i)\rvert}\sum_{\chi\in\Irr(B)}\chi(g)\chi(h_i)=\sum_{i=1}^s\frac{\lvert\C_G(h)\rvert}{\lvert\C_T(h_i)\rvert}\mu_{\id}(g,h_i)\]
where $h_1,\ldots,h_s$ represent the $T$-classes contained in the $G$-class of $h$ and $\id$ is the identical perfect isometry on $\ZZ\Irr(B)$. Similar results hold for the Glauberman correspondence~\cite{HorimotoGL,WatanabeGL}, the Isaacs correspondence~\cite{Sanus}, the Dade correspondence~\cite{TasakaDade,WatanabeD}, Shintani descent~\cite{KessarShintani} and so on.

\item Enguehard~\cite{EnguehardSn} showed that two $p$-blocks of (possibly different) symmetric groups are perfectly isometric whenever they have the same weight. A similar statement for alternating groups was proved in Brunat--Gramain~\cite{BrunatGramain}. 

\item\label{nil} If $B$ is nilpotent with defect group $D$, then $B$ is perfectly isometric to the principal block of $D$ via the Broué--Puig~\cite{BrouePuig} $*$-construction $\Irr(D)\to\Irr(B)$, $\lambda\mapsto\lambda*\chi$ where $\chi\in\Irr(B)$ is a fixed irreducible character of height $0$ (see Section~\ref{sec4}).

\item According to Broué~\cite{BroueEqui}, there are stronger equivalences:
\[\textnormal{Morita equivalence }\Longrightarrow\textnormal{ derived equivalence }\Longrightarrow\textnormal{ perfect isometry}\]
\end{enumerate}
\end{Ex}

\section{Choice of signs}

In the following we denote the ring of algebraic integers in $\CC$ by $\mathbf{R}$. Let $M$ be a maximal ideal of $\mathbf{R}$ containing $p\mathbf{R}$. Then $F:=\mathbf{R}/M$ is an algebraically closed field of characteristic $p$ (see \cite[Lemma~2.1]{Navarro}). Let $\mathcal{O}:=\{r/s:r\in\mathbf{R},\,s\in\mathbf{R}\setminus M\}$. Then $\mathcal{O}/\J(\mathcal{O})\cong F$ and we denote the natural epimorphism by $^*:\mathcal{O}\to F$ (here we differ from \cite[p. 16]{Navarro} where this ring is denoted by $S$). Note that (int) states that $\mu_I(g,h)/\lvert\C_G(g)\rvert,\mu_I(g,h)/\lvert\C_H(h)\rvert\in \mathcal{O}$.

The following lemma is usually not covered in text books. For the convenience of the reader we provide a proof.

\begin{Lem}[Osima~{\cite[Theorem~3]{Osima}}]\label{osima}
Let $J\subseteq\Irr(G)$ such that
\[\sum_{\chi\in J}{\chi(g)\chi(h)}=0\qquad\forall g\in G^0, h\in G\setminus G^0.\]
Then $J$ is a union of blocks.
\end{Lem}
\begin{proof}
We fix $g\in G^0$. Then, by \cite[Theorem~2.13]{Navarro}, there are complex numbers $a_\phi^g$ such that
\[\sum_{\chi\in J}{\chi(g)\chi}=\sum_{\phi\in\IBr(G)}{a_\phi^g\Phi_\phi}.\]
By \cite[Corollary~2.14]{Navarro}, $\Phi_\phi(1)$ is divisible by $|G|_p$ for every $\phi\in\IBr(G)$.
Moreover, \cite[Lemma~2.15]{Navarro} implies that 
\[a_\phi^g=\Bigl[\sum_{\mu\in\IBr(G)}{a_\mu^g\Phi_\mu},\phi\Bigr]^0=\Bigl[\sum_{\chi\in J}{\chi(g)\chi},\phi\Bigr]^0=\sum_{\chi\in J}{\chi(g)[\chi,\hat{\phi}]}\in\mathbf{R}.\]
We conclude that
\[\sum_{\chi\in J}{e_\chi}=\frac{1}{|G|}\sum_{\chi\in J}\chi(1)\sum_{g\in G}{\chi(g^{-1})g}=\sum_{g\in G^0}\Bigl(\sum_{\chi\in J}\frac{\chi(1)\chi(g^{-1})}{|G|}\Bigr)g=\sum_{g\in G^0}\Bigl(\sum_{\phi\in\IBr(G)}\frac{a_\phi^{g^{-1}}\Phi_\phi(1)}{|G|}\Bigr)g\in \Z(\mathcal{O}G).\]
Now the claim follows from \cite[Theorem~3.9]{Navarro}.
\end{proof}

The following is taken from \cite[Lemma~3.2.3]{Ruengrot}.

\begin{Prop}\label{sign}
If $I,J:\ZZ\Irr(B)\to\ZZ\Irr(B')$ are perfect isometries such that $I(\chi)=\pm J(\chi)$ for all $\chi\in\Irr(B)$, then 
$I=\pm J$.
\end{Prop}
\begin{proof}
By \autoref{equiv}, $J^{-1}I:\ZZ\Irr(B)\to\ZZ\Irr(B)$ is a perfect isometry sending $\chi\in\Irr(B)$ to $\pm\chi$. Let $S^+:=\{\chi\in\Irr(B):J^{-1}I(\chi)=\chi\}$. If $g\in G$ is $p$-regular and $h\in G$ is $p$-singular, then 
\[\sum_{\chi\in S^+}\chi(g)\chi(h)-\sum_{\chi\in\Irr(B)\setminus S^+}\chi(g)\chi(h)=\sum_{\chi\in\Irr(B)}\chi(g)J^{-1}I(\chi)(h)=0=\sum_{\chi\in\Irr(B)}\chi(g)\chi(h)\]
by (sep) and \cite[Corollary~3.7]{Navarro}. Hence, $\sum_{\chi\in S^+}\chi(g)\chi(h)=0$. Now \autoref{osima} implies $S^+\in\{\varnothing,\Irr(B)\}$ and the claim follows.
\end{proof}

\begin{Cor}
The perfect isometries $I:\ZZ\Irr(B)\to\ZZ\Irr(B)$ form a group $\PI(B)$ such that 
\[\PI(B)/\langle -\id\rangle\le\Sym(\Irr(B)).\]
\end{Cor}
\begin{proof}
By \autoref{equiv}, $\PI(B)$ is a group with respect to composition of maps.
For $I\in\PI(B)$ let $\widehat{I}:\Irr(B)\to\Irr(B)$ such that $\widehat{I}(\chi)=\pm I(\chi)$ for $\chi\in\Irr(B)$. Then the map $\PI(B)\to\Sym(\Irr(B))$, $I\mapsto\widehat{I}$ is a group homomorphism with kernel $\langle-\id\rangle$ by \autoref{sign} (note that $-\id$ is indeed a perfect isometry).
\end{proof}

\begin{Ex}\label{exS3}
Not every perfect isometry has a \emph{uniform} sign (in the sense that $I(\Irr(B))=\Irr(B')$ or $I(\Irr(B))=-\Irr(B')$): Let $B$ be the principal $3$-block of the symmetric group $G=S_3$. The character table is given by
\[\begin{array}{c|ccc}
B&1&(12)&(123)\\\hline
\chi_1&1&1&1\\
\chi_2&1&-1&1\\
\chi_3&2&.&-1
\end{array}.\]
Hence, the map $\chi_1\mapsto\chi_1$, $\chi_2\mapsto-\chi_3$ and $\chi_3\mapsto-\chi_2$ induces a perfect isometry and the sign is not uniform. It is easy to see that $\PI(B)$ is isomorphic to the dihedral group $D_{12}$ of order $12$.
\end{Ex}

We take the opportunity to determine the $\PI(B)$ for blocks with cyclic defect groups in general. This confirms a conjecture made in \cite[Conjecture~6.0.6]{Ruengrot} (the easy but exceptional cases $e=1$ and $e=|D|-1$ are settled in \cite[Theorem~6.0.5]{Ruengrot}, see also \autoref{abel} below).

\begin{Thm}
Let $B$ be a block with cyclic defect group $D$ and inertial index $e$ such that $1<e<|D|-1$. Then $\PI(B)=\langle-\id\rangle\times S_e\times C_{\phi(|D|)/e}$ where $S_e$ permutes the non-exceptional characters and $C_{\phi(|D|)/e}$ permutes the exceptional characters of $B$ ($\phi$ denotes Euler's totient function).
\end{Thm}
\begin{proof}
By \cite[Théor{\`e}me~5.3]{Broue}, there exists a perfect isometry between $B$ and its Brauer correspondent in $\N_G(D)$ sending exceptional characters to exceptional characters. So we may assume that $D\unlhd G$. It is well-known that the inertial quotient $E$ of $B$ is a $p'$-subgroup of $\Aut(D)$. In particular, $E$ is cyclic of order $e$ dividing $p-1$. Since $e>1$, we conclude that $p$ is odd. By a result of Külshammer~\cite{Kuelshammer} (see \cite[Theorem~1.19]{habil}), we may assume that $B$ is the only block of $G:=D\rtimes E$.
Moreover, $G$ is a Frobenius group and $\Irr(G)=\Irr(E)\cup\{\psi_1^G,\ldots,\psi_t^G\}$ where $t=(|D|-1)/e$ and $\psi_1,\ldots,\psi_t$ is a set of representatives of the $E$-orbits on $\Irr(D)\setminus\{1_D\}$. Since $\Aut(D)$ is cyclic of order $\phi(|D|)$, there exists a unique (cyclic) subgroup $A\le\Aut(D)$ of order $\phi(|D|)/e$. Then $A$ acts on $G$ and permutes $\psi_1^G,\ldots,\psi_t^G$ faithfully. Hence by \autoref{ex1}\eqref{ex1a}, $A$ induces a subgroup of $\PI(B)$ which acts trivially on $\Irr(E)$.  

Now we show that every $I\in\Sym(\Irr(E))\subseteq\Sym(\Irr(G))$ induces a perfect isometry. Observe that $D\setminus\{1\}$ is the set of $p$-singular elements of $G$. Let $g\in G$ be $p$-singular and $h\in G$ $p$-regular. Then $\chi(g)=1$ for $\chi\in\Irr(E)$ and $\psi_i^G(h)=0$ for $i=1,\ldots,t$. Hence,
\[\mu_I(g,h)=\sum_{\chi\in\Irr(E)}I(\chi)(h)=\sum_{\chi\in\Irr(E)}\chi(h)=0,\]
i.\,e. (sep) holds. In order to show (int'), let $g,h\in D\setminus\{1\}$. Then
\[\mu_I(g,h)=e+\sum_{i=1}^t\psi_i^G(g)\psi_i^G(h)=\sum_{\chi\in\Irr(G)}\chi(g)\chi(h)\equiv 0\pmod{|D|}\]
by the second orthogonality relation. Hence, (int') holds and $I$ is a perfect isometry. Consequently, 
\[\langle-\id\rangle\times S_e\times C_{\phi(|D|)/e}\le\PI(B).\] 

Now let $I\in\PI(B)$ be arbitrary. Since $e<|D|-1$, we may choose $1\le i<j\le t$. The generalized character $\psi_i^G-\psi_j^G$ vanishes on the $p$-regular elements of $G$. By \eqref{alter}, $I(\psi_i^G)-I(\psi_j^G)$ also vanishes on the $p$-regular elements. Since $I(\psi_i^G)\ne I(\psi_j^G)$, it follows easily that 
\[\{I(\psi_1^G),\ldots,I(\psi_t^G)\}=\pm\{\psi_1^G,\ldots,\psi_t^G\}.\]
Consequently, $I(\chi)\in\pm\Irr(E)$ for every $\chi\in\Irr(E)$.
Since $e>1$, we may choose distinct $\chi,\psi\in\Irr(E)$. Then $\chi-\psi$ and $I(\chi)-I(\psi)$ vanish on the $p$-singular elements and we obtain $I(\Irr(E))=\pm\Irr(E)$. By the first part of the proof, we may assume that $I(\chi)=\chi$ for every $\chi\in\Irr(E)$. 
Suppose that $I$ has a negative sign on the characters $\psi_i^G$. Then
\[\mu_I(1,1)=e-\sum_{i=1}^te^2=e(1-te)=e(2-|D|)\]
is not divisible by $|D|$ and this contradicts (int). Hence, $I$ has a uniform positive sign.

We consider the column vector $v:=(I(\chi)(g):\chi\in\Irr(G))$ for a fixed generator $g$ of $D$. 
Recall that all character values lie in the cyclotomic field $\QQ_{|G|}$.
By linear algebra over that field, we may write $v$ as a linear combination $v=\alpha_1u_1+\ldots+\alpha_nu_n$ where $\alpha_1,\ldots,\alpha_n\in\QQ_{|G|}$ and $u_1,\ldots,u_n$ are columns of the character table of $G$. 
By the second orthogonality relation, we have
\[\alpha_i(u_i,u_i)=(v,u_i)\]
for $i=1,\ldots,n$ where $(u_i,u_i)$ denotes the usual inner product. 
If some $u_i$ corresponds to a $p$-regular element, then (sep) implies $\alpha_i=(v,u_i)=0$. Hence, we may assume that $u_1,\ldots,u_n$ correspond to $p$-singular elements. 
Again by the second orthogonality relation, we obtain
\[|D|=\lvert\C_G(g)\rvert=(v,v)=\sum_{i=1}^n|\alpha_i|^2(u_i,u_i)=|D|\sum_{i=1}^n|\alpha_i|^2\]
and $|\alpha_1|^2+\ldots+|\alpha_n|^2=1$.
Moreover, (int) implies that
\[\alpha_i=\frac{1}{|D|}(v,u_i)\]
is an algebraic integer for $i=1,\ldots,n$. Since $\QQ_{|G|}$ is an abelian number field, we also get
\[|\alpha_1^\sigma|^2+\ldots+|\alpha_n^\sigma|^2=1\]
for every Galois automorphism $\sigma$ of $\QQ_{|G|}$. In particular, $|\alpha_i^\sigma|\le1$. By Galois theory, the product $\prod_\sigma|\alpha_i^\sigma|$ is a rational integer and we conclude that $|\alpha_i|=1$ for some $i$ and $\alpha_j=0$ for all $j\ne i$. By comparing the first entry (corresponding to the trivial character) of $v=\alpha_iu_i$, we see that $\alpha_i=1$, i.\,e. $v$ is a column of the character table of $G$. Now it is easy to see that $I$ is induced from the automorphism group $A$ introduced above (note that only $\phi(|D|)/e$ columns of the character table contain a primitive $|D|$-th root of unity). Therefore, we have shown that $\PI(B)\le\langle-\id\rangle\times S_e\times C_{\phi(|D|)/e}$.
\end{proof}

\section{Preserved invariants}\label{sec4}

Recall that the \emph{height} $h(\chi)\ge 0$ of $\chi\in\Irr(B)$ is defined by $\chi(1)_p=p^{a-d+h(\chi)}$ where $d$ is the defect of $B$ and $|G|_p=p^a$. Let $\Irr_i(B):=\{\chi\in\Irr(B):h(\chi)=i\}$ and $k_i(B):=\lvert\Irr_i(B)\rvert$.
We show first that the decomposition matrix encodes the character heights.

\begin{Lem}[Brauer~{\cite[5H]{BrauerBlSec2}}]\label{lembrauer}
Let $d$ be the defect and $Q\in\ZZ^{k(B)\times l(B)}$ be the decomposition matrix of $B$. Let 
\[(m_{\chi\psi})_{\chi,\psi\in\Irr(B)}:=p^dQ(Q^\text{t}Q)^{-1}Q^\text{t}\in\ZZ^{k(B)\times k(B)}.\] 
If $\chi\in\Irr_0(B)$ and $\psi\in\Irr_i(B)$, then $(m_{\chi\psi})_p=p^i$.
\end{Lem}
\begin{proof}
Since $C:=Q^\text{t}Q$ is the Cartan matrix of $B$, \cite[Theorem~3.26]{Navarro} shows that $m_{\chi\psi}\in\ZZ$ for $\chi,\psi\in\Irr(B)$. By \cite[Theorem~2.13]{Navarro}, $C^{-1}=([\phi,\mu]^0)_{\phi,\mu\in\IBr(B)}$. Let $|G|_p=p^a$. Then
\[m_{\chi\psi}=p^d\sum_{\phi,\mu\in\IBr(B)}[\phi,\mu]^0d_{\chi\phi}d_{\psi\mu}=p^d[\chi,\psi]^0=p^{d-a}[\widetilde{\chi},\psi]\]
and \cite[Theorem~3.24]{Navarro} yields $(m_{\chi\psi})_p=p^{d-a}\psi(1)_p=p^{h(\psi)}$.
\end{proof}

\begin{Thm}[{Broué~\cite[Théor{\`e}me~1.5]{Broue}}]\label{decomp}
Let $B$ and $B'$ be perfectly isometric blocks with decomposition matrices $Q$ and $Q'$ respectively. Then there exist $S\in\GL(l(B),\ZZ)$ and a signed permutation matrix $T\in\GL(k(B),\ZZ)$ such that \[QS=TQ'.\] 
In particular, $k_i(B)=k_i(B')$ for $i\ge 0$ and $l(B)=l(B')$. Moreover, the Cartan matrices of $B$ and $B'$ are equivalent as integral quadratic forms. In particular, they have the same elementary divisors counting multiplicities. Finally, $B$ and $B'$ have the same defect.
\end{Thm}
\begin{proof}
Let $I:\ZZ\Irr(B)\to\ZZ\Irr(B')$ be a perfect isometry. By \eqref{alter}, $I(\Phi_\phi)$ is a generalized character of $B'$ which vanishes on the $p$-singular elements of $H$. Hence, $I(\Phi_\phi)=\sum_{\mu\in\IBr(B')}s_{\mu\phi}\Phi_\mu$ with $s_{\mu\phi}\in\ZZ$ (see \cite[Corollary~2.17]{Navarro}). This shows
\[\sum_{\chi\in\Irr(B)}d_{\chi\phi}I(\chi)=I(\Phi_\phi)=\sum_{\mu\in\IBr(B')}s_{\mu\phi}\Phi_\mu=\sum_{\chi\in\Irr(B)}\Bigl(\sum_{\mu\in\IBr(B')}d_{I(\chi)\mu}s_{\mu\phi}\Bigr)I(\chi).\]
Setting $S:=(s_{\mu\phi})\in\ZZ^{l(B')\times l(B)}$ and $T=(\epsilon_I(\chi)\delta_{\widehat{I}(\chi)\psi})\in\ZZ^{k(B)\times k(B')}$ where $I(\chi)=\epsilon_I(\chi)\widehat{I}(\chi)$, it follows that $Q=TQ'S$. 

Since $I^{-1}$ is also a perfect isometry, we get matrices $S'\in\ZZ^{l(B)\times l(B')}$ and $T'\in\ZZ^{k(B')\times k(B)}$ such that $Q'=T'QS'$. In fact, by the definition we see that $T'=T^\text{t}=T^{-1}$. Thus, $Q=TQ'S=TT'QS'S=QS'S$ and $S'=S^{-1}\in\GL(l(B),\ZZ)$, because $Q$ has full rank as is well-known. In particular, $l(B)=l(B')$. In accordance with the statement of the theorem, we replace $S$ by $S^{-1}$. 
Then the Cartan matrices of $B$ and $B'$ are given by $C:=Q^{\text{t}}Q$ and 
\[C':=(Q')^{\text{t}}Q'=(T^{-1}QS)^{\text{t}}T^{-1}QS=S^\text{t}CS,\] 
since $T$ is orthogonal. Hence, $C$ and $C'$ are equivalent as integral quadratic forms. We conclude that $C$ and $C'$ have the same elementary divisors counting multiplicities. In particular, the largest elementary divisors of $C$ and $C'$ coincide and this number is the order of a defect group of $B$ and $B'$. So $B$ and $B'$ have the same defect $d$. 

Finally, the claim $k_i(B)=k_i(B')$ follows from 
\[Q'(C')^{-1}(Q')^{\text{t}}=T^\text{t}QC^{-1}Q^{\text{t}}T\]
and \autoref{lembrauer}.
\end{proof}

\begin{Ex}\label{expgrp}\hfill
\begin{enumerate}[(i)]
\item\label{eins} There exist perfectly isometric blocks with non-isomorphic defect groups: 
Let $G$ and $H$ be any $p$-groups with the same character table (like $D_8$ and $Q_8$). Then there exist bijections
$I:\Irr(G)\to\Irr(H)$ and $\sigma:G\to H$ such that $I(\chi)(\sigma(g))=\chi(g)$ for $\chi\in\Irr(B)$ and $g\in G$. By the second orthogonality relation, $I$ induces a perfect isometry between the principal $p$-blocks of $G$ and $H$ (these are of course the only blocks of $G$ and $H$ respectively). We will see in \autoref{hertweck} that the converse holds as well whenever $G$ and $H$ are $p$-groups.

\item Külshammer--Olsson--Robinson~\cite[Section~1]{KOR} introduced a \emph{generalized perfect isometry} \[I:\ZZ\Irr(B)\to\ZZ\Irr(B')\] 
by requiring only $[\chi,\psi]^0=[I(\chi),I(\psi)]^0$ for all $\chi,\psi\in\Irr(B)$. 
This turns out to be equivalent to $QS=TQ'$ with the notation of \autoref{decomp}. We will see in \autoref{abel} below that not every generalized perfect isometry is a perfect isometry in the sense of \autoref{broue}. Other variations of perfect isometries were given by Narasaki--Uno\cite{NarasakiUno}, Eaton~\cite{EatonPI} and Evseev~\cite{Evseev}.
\end{enumerate}
\end{Ex}

\begin{Thm}[{Broué~\cite[Théor{\`e}me~5.2]{Broue}}]\label{center}
Let $B$ and $B'$ be perfectly isometric blocks. Then the centers $\Z(B)$ and $\Z(B')$ are isomorphic as $\mathcal{O}$-algebras and as $F$-algebras.
\end{Thm}
\begin{proof}
Let $I:\ZZ\Irr(B)\to\ZZ\Irr(B')$ be a perfect isometry. We define a linear map
\begin{align*}
\Gamma:\CC G&\to\CC H\\
\sum_{g\in G}\alpha_gg&\mapsto\sum_{h\in H}\Bigl(\frac{1}{|G|}\sum_{g\in G}\mu_I(g,h^{-1})\alpha_g\Bigr)h.
\end{align*}
Since $\mu_I$ is a class function, we see that $\Gamma$ maps into $\Z(\CC H)$.
Setting $\gamma_\chi:=\frac{|H|\chi(1)}{|G|I(\chi)(1)}$ we obtain
\[\Gamma(e_\chi)=\sum_{h\in H}\Bigl(\frac{1}{|G|^2}\chi(1)\sum_{g\in G}\mu_I(g,h^{-1})\chi(g^{-1})\Bigr)h=\sum_{h\in H}\Bigl(\frac{1}{|G|}\chi(1)I(\chi)(h^{-1})\Bigr)h=\gamma_\chi e_{I(\chi)}\]
via \eqref{alter}. 
Now suppose that $\sum_{g\in G}\alpha_gg\in\Z(\mathcal{O}G)$. Let $\mathcal{R}$ be a set of representatives for the conjugacy classes of $G$. Then, by (int),
\begin{equation}\label{coeff}
\frac{1}{|G|}\sum_{g\in G}\mu_I(g,h^{-1})\alpha_g=\sum_{g\in \mathcal{R}}\frac{\mu_I(g,h^{-1})}{\lvert\C_G(g)\rvert}\alpha_g\in \mathcal{O}
\end{equation}
for $h\in H$ and we see that $\Gamma:\Z(\mathcal{O}G)\to\Z(\mathcal{O}H)$. The primitive block idempotent of $B$ over $\mathcal{O}$ is given by $f_B:=\sum_{\chi\in\Irr(B)}e_\chi\in\Z(\mathcal{O}G)$ (see \cite[p. 53]{Navarro}). Hence, $\Gamma:\Z(\mathcal{O}Gf_B)\to\Z(\mathcal{O}Hf_{B'})$.
Since also $I^{-1}$ is a perfect isometry, there exists a similar map $\Lambda:\Z(\mathcal{O}Hf_{B'})\to\Z(\mathcal{O}Gf_B)$ sending $e_{I(\chi)}$ to $\gamma_\chi^{-1} e_\chi$ (when extended to $\Z(\CC H)$). Finally, we define a linear map
\begin{align*}
\Phi:\Z(\mathcal{O}Gf_B)&\to\Z(\mathcal{O}Hf_{B'}),\\
x &\mapsto\Gamma(x\Lambda(f_{B'})).
\end{align*}
For $x=\sum_{\chi\in\Irr(B)}\alpha_\chi e_\chi\in\Z(\mathcal{O}Gf_B)$ with $\alpha_\chi\in\CC$ we obtain
\[\Phi(x)=\Gamma\Bigl(\sum_{\chi\in\Irr(B)}\alpha_\chi\gamma_\chi^{-1} e_\chi\Bigr)=\sum_{\chi\in\Irr(B)}\alpha_\chi e_{I(\chi)}.\]
It follows easily that $\Phi:\Z(B)\to\Z(B')$ is an isomorphism of $\mathcal{O}$-algebras.

Going over to $F$, we denote the block idempotent by $e_B:=f_B^*$ (see \cite[p. 55]{Navarro}). If $x,y\in\Z(\mathcal{O}Gf_B)$ such that $x^*=y^*$, then $\Gamma(x)^*=\Gamma(y)^*$ by \eqref{coeff}. Therefore, $\Phi$ induces a well-defined bijection $\Z(FGe_B)\to\Z(FHe_{B'})$ of $F$-algebras.
\end{proof}

One can show that the isomorphism in \autoref{center} also preserves the projective center (see \cite[Theorem~4.11]{BroueEqui}).

\begin{Prop}[{Broué~\cite[Lemme~1.6]{Broue}}]\label{degcorr}
If $I:\ZZ\Irr(B)\to\ZZ\Irr(B')$ is a perfect isometry, then
\[(I(\chi)(1)\psi(1))_{p'}\equiv (I(\psi)(1)\chi(1))_{p'}\pmod{p}\]
for all $\chi,\psi\in\Irr(B)$.
\end{Prop}
\begin{proof}
In the proof of \autoref{center} we have constructed a linear map $\Gamma$ sending $e_{\chi}$ to $\gamma_\chi e_{I(\chi)}$ where $\gamma_\chi=\frac{|H|\chi(1)}{|G|I(\chi)(1)}$. Since $B$ and $B'$ have the same defect and $I$ preserves character heights, it follows that $\gamma_\chi\in\mathcal{O}^\times$. For a fixed $\chi\in\Irr(B)$ we have
\[A:=\sum_{\psi\in\Irr(B)\setminus\{\chi\}}(\gamma_\psi-\gamma_\chi)e_{I(\psi)}=\sum_{\psi\in\Irr(B)}\gamma_\psi e_{I(\psi)}-\sum_{\psi\in\Irr(B)}\gamma_\chi e_{I(\psi)}=\Gamma(f_B)-\gamma_\chi f_{B'}\in\Z(\mathcal{O}G).\]
Hence, $(\gamma_\psi-\gamma_\chi)^*=\lambda_\psi(A^*)=\lambda_\chi(A^*)=0$ and 
$\gamma_\chi^*=\gamma_\psi^*$ for $\chi,\psi\in\Irr(B)$ (cf. \cite[Theorem~3.9]{Navarro}). The claim follows.
\end{proof}

It was conjectured in \cite[Conjecture~4.1.13]{Ruengrot} that $\langle -\id\rangle$ always has a complement in $\PI(B)$. This was verified in \cite[Proposition~4.1.12]{Ruengrot} whenever $k(B)$ is odd. We remark that the conjecture holds more generally if some $k_i(B)$ is odd. In fact, in this case the set of perfect isometries $I\in\PI(B)$ such that \[|\{\chi\in\Irr_i(B):I(\chi)\in-\Irr_i(B)\}|\equiv 0\pmod{2}\] 
forms a complement of $\langle-\id\rangle$. Moreover, the map sending the signed permutation matrix $T$ in \autoref{decomp} to $S$ induces a homomorphism $\PI(B)\to\GL(l(B),\ZZ)$. If $l(B)$ is odd, then the preimage of $\SL(l(B),\ZZ)$ under this map forms again a complement of $\langle-\id\rangle$.

\section{Nilpotent blocks}

As a motivation, we start with a known result about character tables. We provide a proof for the convenience of the reader (cf. Weidman~\cite{Weidman}, Chillag~\cite{Chillag} and Lux--Pahlings~\cite[Section 2.4]{LuxPahlings}).

\begin{Thm}\label{strucconst}
Let $K_1,\ldots,K_n$ be the conjugacy classes and $K_1^+,\ldots,K_n^+\subseteq\Z(\ZZ G)$ be the class sums of $G$, and let $\Irr(G)=\{\chi_1,\ldots,\chi_n\}$. Then the character table of $G$ is determined \textup{(}up to labeling of rows and columns\textup{)} by each one of the following sets of integers:
\begin{enumerate}[(i)]
\item $a_{ijk}$ such that $K_i^+K_j^+=\sum_{k=1}^na_{ijk}K_k^+$ for $i,j,k\in\{1,\ldots,n\}$.
\item $b_{ijk}$ such that $\chi_i\chi_j=\sum_{k=1}^nb_{ijk}\chi_k$ for $i,j,k\in\{1,\ldots,n\}$.  
\end{enumerate}
\end{Thm}
\begin{proof}\hfill
\begin{enumerate}[(i)]
\item We define the central characters as usual by $\omega_i(K_j^+):=\chi_i(g_j)|K_j|/\chi(1)$ where $g_j\in K_j$. For $i=1,\ldots,n$ set $M_i:=(a_{ijk})_{j,k}$ and $s_i:=(\omega_i(K_j))_j$. 
Then 
\[\omega_l(K_i^+)\omega_l(K_j^+)=\omega_l(K_i^+K_j^+)=\sum_{k=1}^na_{ijk}\omega_l(K_k^+)\]
and $\omega_l(K_i^+)s_l=M_is_l$ for $l=1,\ldots,n$. Since the central characters are linearly independent, we have $S:=(s_l)_l\in\GL(n,\CC)$. Hence, $S^{-1}M_iS=\mathrm{diag}(\omega_1(K_i^+),\ldots,\omega_n(K_i^+))$ for $i=1,\ldots,n$. This means that $M_1,\ldots,M_n$ are simultaneously diagonalizable. Since also the rows of $S$ are linearly independent, it follows that $S$ is uniquely determined by $a_{ijk}$ up to permutations and signs of columns.
One column of $S$ has the form $(\omega_1(K_i^+))_i=(|K_i|)_i$ and this is the only column consisting of positive integers (by the second orthogonality relation). Therefore we obtain the class sizes from $a_{ijk}$. By the first orthogonality relation, we also have
\[\sum_{i=1}^n|K_i|\lvert\chi_j(x_i)\rvert^2=\sum_{g\in G}\lvert\chi_j(g)\rvert^2=|G|[\chi_j,\chi_j]=|G|.\]
This implies that we get the character degrees from $S$. Altogether, the numbers $a_{ijk}$ determine the character table $T$ of $G$ up to signs of rows. In order to show that the signs are irrelevant, assume that $\diag(\epsilon_1,\ldots,\epsilon_n)T$ where $\epsilon_i\in\{\pm1\}$ is also a character table of some finite group. Then there must be some $i$ such that $\chi_j(g_i)=\epsilon_j\chi_j(1)$ for $j=1,\ldots,n$. In particular, $g_i\in\Z(G)$ and the map $\pi:g_j\mapsto g_jg_i$ induces a permutation of $\{g_1,\ldots,g_n\}$. Since $\chi_k(g_jg_i)=\chi_k(g_j)\chi_k(g_i)/\chi_k(1)=\epsilon_k\chi_k(g_j)$, there exists a permutation matrix $Q$ corresponding to $\pi$ such that $\diag(\epsilon_1,\ldots,\epsilon_n)T=TQ$. Hence, $T$ is (essentially) uniquely determined from $a_{ijk}$.

\item We define $M_i:=(b_{ijk})_{j,k}$ and $s_l:=(\chi_i(g_l))_i$ for $i=1,\ldots,n$. Then
\[\sum_{k=1}^nb_{ijk}\chi_k(g_l)=\chi_i(g_l)\chi_j(g_l)\]
and $M_is_l=\chi_i(g_l)s_l$. Hence for the character table $T:=(s_l)_l$ we get 
\[T^{-1}M_iT=\diag(\chi_i(g_1),\ldots,\chi_i(g_n)).\] 
Arguing as in (i), we obtain $T$ from $b_{ijk}$ up to permutations and signs of columns. Suppose that there are signs $\epsilon_1,\ldots,\epsilon_n$ such that $T\diag(\epsilon_1,\ldots,\epsilon_n)$ is also the character table of some finite group. Then there exists $i$ such that $\chi_i(g_j)=\epsilon_j1_G(g_j)=\epsilon_j$ for $j=1,\ldots,n$. Hence, $\chi_i$ is a linear character and we have a permutation $\pi$ on $\Irr(G)$ sending $\chi_j\mapsto\chi_i\chi_j$. It follows that the permutation matrix $Q$ corresponding to $\pi$ satisfies $T\diag(\epsilon_1,\ldots,\epsilon_n)=QT$. Thus, we obtain $T$ from $b_{ijk}$.\qedhere
\end{enumerate}
\end{proof}

Conversely, it is well-known that the character table $T$ of $G$ determines $a_{ijk}$ via
\[a_{ijk}=\frac{|K_i||K_j|}{|G|}\sum_{l=1}^n\frac{\chi_l(g_i)\chi_l(g_j)\chi_l(g_k^{-1})}{\chi_l(1)}\]
(see \cite[Problem (3.9)]{Isaacs}). Of course, $T$ also determines $b_{ijk}=[\chi_i\chi_j,\chi_k]$. 
The numbers $a_{ijk}$ and $b_{ijk}$ are the structure constants of the $\ZZ$-algebras $\Z(\ZZ G)$ and $\ZZ\Irr(G)$ respectively. We remark that these algebras are in general not isomorphic ($Q_8$ is a counterexample as can be seen by reducing modulo $2$). On the other hand, $\Z(\CC G)\cong\CC\Irr(G)\cong\CC^n$ where $n$ is the class number of $G$. 

The following observation relies on a result of Hertweck~\cite{Hertweck}. It is also related to the work of Zhou--Sun~\cite{ZhouSun}.

\begin{Thm}\label{hertweck}
Let $B$ and $B'$ be nilpotent with defect groups $P$ and $Q$ respectively. Then the following statements are equivalent:
\begin{enumerate}[(i)]
\item\label{one} $B$ and $B'$ are perfectly isometric.
\item\label{two} $\Z(B)$ and $\Z(B')$ are isomorphic $\mathcal{O}$-algebras.
\item\label{three} $P$ and $Q$ have the same character table (up to labeling of rows and columns).
\end{enumerate}
In this case, every perfect isometry between $B$ and $B'$ has a uniform sign. In particular, 
\[\PI(B)\le\langle-\id\rangle\times\prod_{i\ge 0}\Sym(\Irr_i(B)).\]
\end{Thm}
\begin{proof}
The implication \eqref{one}$\Rightarrow$\eqref{two} follows from \autoref{center}. 
Recall from \autoref{ex1}\eqref{nil} that $B$ (respectively $B'$) is perfectly isomorphic to the principal block of $P$ (respectively $Q$). Thus, to prove \eqref{two}$\Rightarrow$\eqref{three} we may assume that $G=P$ and $H=Q$. 
In the notation of \cite[Section~3]{Hertweck}, $\mathcal{O}$ is a $G$-adapted integral domain. Hence, \cite[Theorem~4.2 and Remark 3.4]{Hertweck} implies that $P$ and $Q$ have the same character table. Finally, the implication \eqref{three}$\Rightarrow$\eqref{one} follows from \autoref{expgrp} and \autoref{ex1}\eqref{nil}.

For the second claim we note first that the $*$-construction gives a perfect isometry between $B$ and the principal block of $P$ with positive signs. Since the same is true of $B'$ and $Q$, we may again assume that $G=P$ and $H=Q$. 
By \autoref{decomp}, every perfect isometry $I:\ZZ\Irr(B)\to\ZZ\Irr(B')$ preserves character heights. Consequently,  $I(\chi)(1)=\pm\chi(1)$ for all $\chi\in\Irr(B)$. Hence by (sep),
\begin{align*}
\sum_{\chi\in\Irr(B)}\pm\chi(1)^2&=\mu_I(1,1)=\sum_{h\in H}\mu_I(1,h)=\sum_{\chi\in\Irr(B)}\chi(1)\sum_{h\in H}I(\chi)(h)\\
&=I^{-1}(1_H)(1)|H|=\pm|G|=\pm\sum_{\chi\in\Irr(B)}\chi(1)^2
\end{align*}
and the second claim follows. The last claim follows from \autoref{sign}.
\end{proof}

In the situation of \autoref{hertweck} it is possible to compute $\PI(B)$ efficiently from the character table of $P$. For this we may assume as usual that $G=P$ and $I\in\PI(B)$ has positive sign. Since multiplication with a linear character induces a perfect isometry (\autoref{ex1}\eqref{lin}), we may assume that $I(1)=1$ where $1$ is the trivial character. We call such a perfect isometry \emph{normalized}. Following the proof of \autoref{center}, we see that $I$ induces the automorphism $\Gamma$ on $\Z(\mathcal{O}G)$ sending $e_\chi$ to $e_{I(\chi)}$ (when extended to $\Z(\CC G)$). In particular,
\[\sum_{g\in G}g=|G|e_1=\Gamma(|G|e_1)=\Gamma\Bigl(\sum_{g\in G}g\Bigr).\]
Hence, by \cite[Theorem~3.2]{Hertweck}, $\Gamma$ maps class sums to class sums (not just scalar multiples of class sums). Let $g_1,\ldots,g_n\in G$ be a set of representatives for the conjugacy classes of $G$.
Let $\sigma\in S_n$ such that the class sum of $g_i$ is mapped to the class sum of $g_{\sigma(i)}$ under $\Gamma$. 
Then the definition of $\Gamma$ shows that $\mu_I(g_i,g_j^{-1})=\delta_{\sigma(i),j}\lvert\C_G(g_i)\rvert$ for $i,j\in\{1,\ldots,n\}$ with the Kronecker delta. Now \eqref{alter} yields 
\begin{equation}\label{tableaut}
I(\chi)(g_{\sigma(i)})=\chi(g_i^{-1}).
\end{equation}
Conversely, if $I\in\Sym(\Irr(G))$ and $\sigma\in S_n$ are given such that $I(\chi)(g_{\sigma(i)})=\chi(g_i^{-1})$, then one easily checks that $I$ induces a normalized perfect isometry (we note that $I$ is not necessarily a so-called \emph{table automorphism} which additionally has to preserve power maps). For a linear character $\lambda\in\Irr(G)$ we have $I(\chi\lambda)=I(\chi)I(\lambda)$ by \eqref{tableaut}. This shows that $\PI(B)$ contains the normal subgroup $\langle-\id\rangle\times\Irr(P/P')\cong C_2\times P/P'$ and we construct the \emph{normalized perfect isometry group}
\[\overline{\PI(B)}:=\PI(B)/(\langle-\id\rangle\times\Irr(P/P')).\]
This quotient can be computed as a subgroup of $\Sym(\Irr(P))$ in GAP~\cite{GAP48} conveniently via \[\texttt{TransformingPermutations(Irr(P),Irr(P)).group}.\]
Now we can generalize the main result of \cite{Ruengrot2}.

\begin{Cor}\label{abel}
If $B$ and $B'$ are perfectly isometric nilpotent blocks with abelian defect groups $P$ and $Q$ respectively, then $P\cong Q$ and $\PI(B)\cong C_2\times(P\rtimes\Aut(P))$.
\end{Cor}
\begin{proof}
Since the isomorphism type of an abelian group is determined by its character table, the first claim is a consequence of \autoref{hertweck}. For the second claim we may assume that $G=P$ and $I\in\PI(B)$ is normalized. Then \eqref{tableaut} shows that $I(\lambda\mu)=I(\lambda)I(\mu)$ for $\lambda,\mu\in\Irr(B)$. Hence, $I$ is induced from $\Aut(G)$. Now the second claim follows easily.
\end{proof}

In general the homomorphism
\[\Psi:\Aut(P)\times\Gal(\QQ_{|P|}|\QQ)\to\overline{\PI(B)}\]
coming from \autoref{ex1} is not surjective. For instance, $\Aut(Q_8)\cong S_4$ induces an element of order $3$ in $\overline{\PI(Q_8)}\cong\overline{\PI(D_8)}$ which cannot be induced from $\Aut(D_8)\cong D_8$. In fact, $\PI(D_8)\cong S_4\times C_2$. Nevertheless, $\Psi(\Gal(\QQ_{|P|}|\QQ))\subseteq\Z(\overline{\PI(B)})$ by \eqref{tableaut}. This verifies the Galois refinement of the Alperin--McKay conjecture for nilpotent blocks (see \cite{NavarroCon}).

According to \cite{Hertweck}, it is conjectured that for any finite groups $G$ and $H$ every (normalized) isomorphism $\Z(\ZZ G)\to\Z(\ZZ H)$ sends class sums to class sums. This is not true anymore when $\ZZ$ is replaced by $\mathcal{O}$ as can be seen from \autoref{exS3}.
Cliff~\cite{Cliff} has constructed blocks with isomorphic centers over $F$ (but not over $\mathcal{O}$) which are not perfectly isometric. 

\section{Generalized decomposition matrices}

In the case of non-nilpotent blocks, the generalized decomposition matrix is some sort of replacement of the character table of the defect group. Recall that a \emph{$B$-subsection} is a pair $(u,b)$ where $u\in G$ is a $p$-element and $b$ is a Brauer correspondent of $B$ in $\C_G(u)$. A \emph{basic set} for $b$ is a basis of the $\ZZ$-module $\ZZ\IBr(b)$ (in particular, $\IBr(b)$ is a basic set). 

The following theorem states that blocks are perfectly isometric if they have the same generalized decomposition matrix up to basic sets (see \cite[p. 133]{Navarro}). 
This generalizes a result of Horimoto--Watanabe~\cite[Theorem~2]{WatanabePerfIso} (the hypothesis (i) in their paper is superfluous). A different generalization has been given in Watanabe~\cite[Theorem~2]{WatanabePP}.

\begin{Thm}\label{HW}
Let $S(B)$ \textup{(}resp. $S(B')$\textup{)} be a set of representatives for the $G$-conjugacy classes of $B$-subsections.
For $(u,b)\in S(B)$ let 
\[Q_{(u,b)}:=(d_{\chi\phi}^u)_{\substack{\chi\in\Irr(B)\\\phi\in\IBr(b)}}\in\CC^{k(B)\times l(b)}\] 
be the generalized decomposition matrix with respect to $(u,b)$. Suppose that there exist a signed permutation matrix $T\in\GL(k(B),\ZZ)$ and a bijection $S(B)\to S(B')$, $(u,b)\mapsto(u',b')$ such that $(1_G',B')=(1_H,B')$ and for every $(u,b)\in S(B)\setminus\{(1,B)\}$ we have 
\begin{equation}\label{gendecrel}
Q_{(u,b)}S_{(u,b)}=TQ_{(u',b')}
\end{equation}
for some $S_{(u,b)}\in\GL(l(b),\ZZ)$.
Then $B$ and $B'$ are perfectly isometric.
\end{Thm}
\begin{proof}
We first show that $Q_{(1,B)}S_{(1,B)}=TQ_{(1,B')}$ holds for some $S_{(1,B)}\in\GL(l(B),\ZZ)$ (this is necessary by \autoref{decomp}). By the orthogonality relations for generalized decomposition numbers (\cite[Lemma~5.13]{Navarro}), the columns of the ordinary decomposition matrix $Q_{(1,B)}$ are orthogonal to the columns of $Q_{(u,b)}$ where $(u,b)\in S(B)\setminus\{(1,B)\}$. Moreover, it is well-known that $Q_{(1,B)}$ has a left inverse (\cite[Lemma~3.16]{Navarro}). It follows that the columns of $Q_{(1,B)}$ form a $\ZZ$-basis for the orthogonal space of the columns of all $Q_{(u,b)}$ with $(u,b)\ne(1,B)$. By the given relations \eqref{gendecrel}, it is clear that the columns of $TQ_{(1,B)}$ form a basis for the corresponding orthogonal space of $B'$. This implies the existence of $S_{(1,B)}$.

The equations \eqref{gendecrel} imply that $l(b)=l(b')$ and $k(B)=k(B')$. 
Let 
\[T=(\epsilon_\chi\delta_{\widehat{I}(\chi)\psi})_{\substack{\chi\in\Irr(B)\\\psi\in\Irr(B')}}\] 
for $\epsilon_\chi\in\{\pm1\}$ and some bijection $\widehat{I}:\Irr(B)\to\Irr(B')$. 
Of course we define $I:\ZZ\Irr(B)\to\ZZ\Irr(B')$ by $I(\chi):=\epsilon_\chi\widehat{I}(\chi)$.

We (may) choose $S(B)$ (and similarly $S(B')$) such that for $(u,b),(v,c)\in S(B)$ we have $u=v$ whenever $u$ and $v$ are conjugate in $G$.
Let 
\begin{align*}
\Bl(u)&:=\{b\in\Bl(\C_G(u)):(u,b)\in S(B)\},\\
S(u')&:=\{(v,c)\in S(B):v'=u'\}.
\end{align*}
Note that we may have $v'=u'$, but $v\ne u$.
For $b\in\Bl(u)$ and $\phi\in\IBr(b)$ we define $\phi':=\sum_{\mu\in\IBr(b')}s_{\phi\mu}\mu\in\ZZ\IBr(b')$ where $S_{(u,b)}=(s_{\phi\mu})$.
Then for $\chi\in\Irr(B)$ and $h\in\C_H(u')^0$ we have 
\begin{align*}
I(\chi)(u'h)&=\sum_{(v,c)\in S(u')}\sum_{\mu\in\IBr(c')}d_{I(\chi)\mu}^{u'}\mu(h)=\sum_{(v,c)\in S(u')}\sum_{\mu\in\IBr(c')}\sum_{\phi\in\IBr(c)}d_{\chi\phi}^vs_{\phi\mu}\mu(h)\\
&=\sum_{(v,c)\in S(u')}\sum_{\phi\in\IBr(c)}d_{\chi\phi}^v\phi'(h).
\end{align*}

Now let $g\in G$ and $h\in H$. 
If $g_p$ is not conjugate to some $u$ with $(u,b)\in S(B)$, then $\chi(g)=0$ for $\chi\in\Irr(B)$ and $\mu_I(g,h)=0$. The same applies to $h$. Hence, we may assume that $u:=g_p$ and $v':=h_p^{-1}$ for some $(u,b),(v,c)\in S(B)$. Then
\begin{align*}
\mu_I(g,h)&=\sum_{\chi\in\Irr(B)}\chi(g)\overline{I(\chi)(h^{-1})}=\sum_{\chi\in\Irr(B)}\sum_{b\in\Bl(u)}\sum_{(w,c)\in S(v')}\sum_{\phi\in\IBr(b)}\sum_{\mu\in\IBr(c)}d_{\chi\phi}^u\overline{d_{\chi\mu}^w}\phi(g_{p'})\mu'(h_{p'})\\
&=\sum_{b\in\Bl(u)}\sum_{(w,c)\in S(v')}\sum_{\phi\in\IBr(b)}\sum_{\mu\in\IBr(c)}\phi(g_{p'})\mu'(h_{p'})\sum_{\chi\in\Irr(B)}d_{\chi\phi}^u\overline{d_{\chi\mu}^w}\\
&=\sum_{\substack{b\in\Bl(u)\\(u,b)\in S(v')}}\sum_{\phi,\mu\in\IBr(b)}c_{\phi\mu}\phi(g_{p'})\mu'(h_{p'})=\sum_{\substack{b\in\Bl(u)\\(u,b)\in S(v')}}\sum_{\mu\in\IBr(b)}\Phi_\mu(g_{p'})\mu'(h_{p'}).
\end{align*}
If exactly one of $g$ and $h$ is $p$-regular, then $\{b\in\Bl(u):(u,b)\in S(v')\}=\varnothing$ and $\mu_I(g,h)=0$ (here we use $1_G'=1_H$). Hence, (sep) holds. Moreover, it follows from \cite[Lemma~2.21]{Navarro} that $\mu_I(g,h)/\lvert\C_G(g)\rvert_p$ is an algebraic integer, since $\C_G(g)=\C_G(u)\cap\C_G(g_{p'})$. 
To prove the second half of (int) we note that the hypothesis is symmetric in $B$ and $B'$. Hence the isometry $I^{-1}$ leads to the algebraic integer $\mu_{I^{-1}}(h,g)/\lvert\C_H(h)\rvert_p$. Recall from \eqref{inverse} that $\mu_{I^{-1}}(h,g)=\mu_I(g,h)$. Thus, the proof is complete.
\end{proof}

\begin{Cor}\label{finitetype}
There are only finitely many perfect isometry classes of $p$-blocks with a given defect.
\end{Cor}
\begin{proof}
This is a consequence of Brauer~\cite[Theorem~8]{BrauerApp1} and \autoref{HW}. For the convenience of the reader we sketch the details. Let $B$ be a block of defect $d$ with subsection $(u,b)$. Let $d_b$ be the defect of $b$, and let $C_b$ be the Cartan matrix of $b$. Then $d_b\le d$ and every elementary divisor of $C_b$ divides $p^{d_b}$.
By the well-known Brauer--Feit bound we have $l(b)\le k(B)\le p^{d^2}$. In particular, $\det C_b$ is bounded in terms of $d$. 
By the reduction theory of quadratic forms, there exist only finitely many equivalence classes of positive definite quadratic forms with given dimension and determinant (discriminant). This means that there exists $S\in\GL(l(b),\ZZ)$ such that all entries of $S^\text{t}C_bS$ are bounded in terms of $d$. The generalized decomposition numbers $d_{\chi\phi}^u$ are algebraic integers in $\QQ_{p^d}$. Since the size of the generalized decomposition matrix $Q_{(u,b)}\in\CC^{k(B)\times l(b)}$ is bounded in terms of $d$, there are only finitely many solutions of the matrix equation
\[Q_{(u,b)}^\text{t}\overline{Q_{(u,b)}}=S^\text{t}C_bS\]
Now \autoref{HW} applies.
\end{proof}

We illustrate \autoref{finitetype} with an example by Kiyota~\cite{Kiyota} which did not appear in this generality before.

\begin{Prop}\label{3x3}
Every $3$-block of defect $2$ is perfectly isometric to one of the following blocks:
\begin{enumerate}[(i)]
\item the principal block of $C_9$ or of $D_{18}$,
\item the principal block of $H\le\AGammaL(1,9)\cong (C_3\times C_3)\rtimes SD_{16}$,
\item the non-principal block of a double cover of $S_3\times S_3$. 
\item the non-principal block of a double cover of $S_3\wr C_2$.
\item\label{last} a non-principal block $B$ with $k(B)=3$.
\end{enumerate}
There are 13 or 14 perfect isometry classes of such blocks depending on whether case \eqref{last} occurs.
\end{Prop}
\begin{proof}
Let $B$ be a block of $G$ with defect group $D$ of order $9$.
In view of \autoref{HW}, it suffices to determine the matrices $Q_{(u,b)}$ up to basic sets. Since this is a tedious task, we will cite some results. 
Let $\beta$ be a Brauer correspondent of $B$ in $\C_G(D)$. Then $T(B):=\N_G(D,\beta)/\C_G(D)\le\Aut(D)$ is the inertial quotient of $B$, and $T(B)$ is a $3'$-group. If $D$ is cyclic, then $|T(B)|\le 2$ and the result follows for instance from Usami~\cite{Usami23I}. Now let $D\cong C_3\times C_3$. Then $\Aut(D)\cong\GL(2,3)$ and $T(B)\le\GammaL(1,9)\cong SD_{32}$ (semidihedral group). By Sambale~\cite[Theorem~3]{SambaleBroue}, it suffices to determine the possible pairs $(k(B),l(B))$. This was done mostly by Kiyota~\cite{Kiyota} and Watanabe~\cite{WatanabeSD16}. In particular, if $T(B)\not\cong C_2\times C_2$ and $|T(B)|\ne 8$, then $B$ is perfectly isometric to the principal block of $D\rtimes T(B)$. For $T(B)\cong C_2\times C_2$ and $T(B)\cong D_8$ there is a second possibility (apart from $D\rtimes T(B)$) given by the non-principal block of a double cover of $D\rtimes T(B)$. Finally, if $T(B)\in\{C_8,Q_8\}$ we have $(k(B),l(B))\in\{(9,8),(6,5),(3,2)\}$. If $k(B)\ne 3$, then $B$ is perfectly isometric to $D\rtimes C_8$ or to $D\rtimes Q_8$ (we do not know if $T(B)$ determines which case occurs). On the other hand, we do not know if $k(B)=3$ can actually occur. An easy analysis of the decomposition numbers shows that there are no principal blocks of that form (see \cite[Proposition~15.7]{habil}).

Counting subgroups $T(B)$ of $\GL(2,3)$ only up to conjugation, we have constructed $13$ (or $14$) perfect isometry classes. It remains to show that these are pairwise not perfectly isometric. 
In most cases the class is uniquely identified by the pair $(k(B),l(B))$ according to \autoref{decomp}. 
However, there are two exceptions. If $(k(B),l(B))=(9,1)$, then $T(B)=1$ and $D\in\{C_9,C_3\times C_3\}$. Here $B$ is nilpotent and \autoref{abel} applies. Now suppose that $(k(B),l(B))=(6,2)$. Then there are three choices: $(D,T(B))\in\{(C_9,C_2),(C_3\times C_3,C_2),(C_3\times C_3,D_8)\}$. As mentioned earlier, a perfect isometry preserves the stable center $\overline{\Z}(B)$ of $B$. In the first two cases $T(B)$ acts semiregularly on $D\setminus\{1\}$ and \cite[Theorem~3.1]{FrobeniusInertial} shows that $\overline{\Z}(B)$ is a symmetric algebra. On the other hand, $D_8$ can never act semiregularly and therefore in the third case $\overline{\Z}(B)$ is not symmetric. Hence, we are left with $|T(B)|=2$ and $D\in\{C_9,C_3\times C_3\}$. These blocks can be distinguished with the character table and we leave the details to the reader.
\end{proof}

Finally we address the converse of \autoref{HW}.
Suppose that $B$ and $B'$ are perfectly isometric. Then we know from \autoref{decomp} that the ordinary decomposition matrices satisfy \eqref{gendecrel} in \autoref{HW}. In the case of nilpotent blocks, also the generalized decomposition matrices satisfy \eqref{gendecrel}, because these matrices form the character table of the defect group and \autoref{hertweck} applies (see \cite[Lemma~10]{SambaleC4}). 

In general, let $Q_*=(Q_{(u,b)})_{(u,b)\in S(B)}\in\CC^{k(B)\times k(B)}$ and similarly $Q'_*$. Let $\widehat{\Z}(B)$ be the ideal of $\Z(B)$ generated by the elements $|G|/\chi(1)e_\chi\in\Z(B)$ ($\chi\in\Irr(B)$). Then by \autoref{center} there exists an isomorphism of $\mathcal{O}$-algebra $\Z(B)\to\Z(B')$ sending $\widehat{\Z}(B)$ to $\widehat{\Z}(B')$. Let $\mathcal{D}_{k(B)}(\mathcal{O})$ be the set of diagonal matrices in $\mathcal{O}^{k(B)\times k(B)}$. By Puig~\cite{Puigcenter}, there exists an isomorphism of $\mathcal{O}$-algebras
\[
Q_*\mathcal{O}^{k(B)\times k(B)}Q_*^{-1}\cap \mathcal{D}_{k(B)}\to Q'_*\mathcal{O}^{k(B)\times k(B)}(Q'_*)^{-1}\cap \mathcal{D}_{k(B)}
\]
sending $Q_*\mathcal{O}^{k(B)\times k(B)}\cap\mathcal{D}_{k(B)}(\mathcal{O})$ to $Q'_*\mathcal{O}^{k(B)\times k(B)}\cap\mathcal{D}_{k(B)}(\mathcal{O})$ (note that Puig uses the transpose of $Q_*$). 
The identity is such an isomorphism whenever $Q_*S=TQ'_*$ for $S\in\GL(k(B),\mathcal{O})$ and $T\in\GL(k(B),\mathcal{O})\cap\mathcal{D}_{k(B)}(\mathcal{O})$. This is of course more general than the matrices $(S_{(u,b)})_{(u,b)\in S(B)}$ and $T$ coming from \autoref{HW}. 
Nevertheless, we do not know if the converse of \autoref{HW} still holds, that is, if perfectly isometric blocks have the “same” generalized decomposition matrices. To fill this lack of knowledge, one replaces perfect isometries by isotypies.

\section{Isotypies}
In order to define isotypies we need to recall some terminology about fusion in blocks which was first introduced by Alperin--Broué~\cite{LocalMethods}. Let $D$ be a defect group of $B$. A \emph{$B$-subpair} is a pair $(Q,b_Q)$ such that $Q\le D$ and $b_Q$ is a Brauer correspondent of $B$ in $\C_G(Q)$. In particular, a subsection $(u,b)$ induces a subpair $(\langle u\rangle,b)$. 
In the case $Q=D$ we speak of \emph{Sylow} subpairs. 
For a given Sylow $B$-subpair $(D,b_D)$ one defines a partial ordering such that for every $Q\le D$ there exists just one $B$-subpair $(Q,b_Q)\le(D,b_D)$ (see \cite[p. 219]{Navarro}).

The following definition is a bit more general than \cite[Définition~4.3]{Broue} in the sense that we do not require that $B$ and $B'$ have the same defect group and fusion system.

\begin{Def}\label{isotypy}
Let $(P,b_P)$ be a Sylow $B$-subpair and let $(Q,b_Q)$ be a Sylow $B'$-subpair. Let $S(B)$ \textup{(}resp. $S(B')$\textup{)} be a set of representatives for the $G$-conjugacy classes of $B$-subsections $(u,b)\in(P,b_P)$. For $(u,b)\in S(B)$ let $d^u:\ZZ\Irr(B)\to\ZZ\IBr(b)$, $\chi\mapsto\sum_{\phi\in\IBr(b)}d^u_{\chi\phi}\phi$. Then $B$ and $B'$ are called \emph{isotypic} if there exists a bijection $S(B)\to S(B')$, $(u,b)\mapsto (u',b')$ such that the following holds:
\begin{itemize}
\item $(1_G',B')=(1_H,B')$,
\item for every $(u,b)\in S(B)$ there exists a perfect isometry $I^u:\ZZ\Irr(b)\to\ZZ\Irr(b')$ such that $I^u\circ d^u=d^{u'}\circ I^1$ (we regard $\ZZ\IBr(b)$ as subset of $\QQ\Irr(b)$ by setting $0$ on the $p$-singular elements),
\item if $\langle u\rangle=\langle v\rangle$, then $I^u=I^v$. 
\end{itemize}
\end{Def}

Note that if $(u,b),(u,c)\in S(B)$ such that $\langle u\rangle=\langle v\rangle$, then $b=c$, since $b$ is the only block such that $(\langle u\rangle,b)\le (P,b_P)$. Hence, the last part of \autoref{isotypy} is meaningful.
Using \autoref{equiv} it is easy to see that isotypy is an equivalence relation.

\begin{Ex}
It is clear that every isotypy between $B$ and $B'$ gives a perfect isometry $I^1:\ZZ\Irr(B)\to\ZZ\Irr(B')$.
Conversely, not every perfect isometry can be extended to an isotypy. Consider for instance nilpotent blocks $B$ and $B'$ with defect group $D_8$ and $Q_8$ respectively. There exists a $B$-subsection $(u,b)$ such that $b$ has defect group $C_2\times C_2$. However there is no such $B'$-subsection. Hence by \autoref{abel}, there is no perfect isometry between $b$ and any $B'$-subsection. 
\end{Ex}

\begin{Prop}[Broué~{\cite[Théor{\`e}me~4.8]{Broue}}]\label{isodec}
If $B$ and $B'$ are isotypic, then the generalized decomposition matrices satisfy 
\[Q_{(u,b)}S_{(u,b)}=TQ_{(u',b')}\]
for every $B$-subsection $(u,b)$ as in \autoref{HW}. 
\end{Prop}
\begin{proof}
For $(u,b)\in S(B)$ and $\chi\in\Irr(B)$ we have 
\begin{equation}\label{stuff}
\sum_{\phi\in\IBr(b)}d_{\chi\phi}^uI^u(\phi)=\sum_{\mu'\in\IBr(b')}d_{I(\chi)\mu'}^{u'}\mu'.
\end{equation}
For $\mu'\in\IBr(b')$ let $\Phi_{\mu'}=\sum_{\psi\in\Irr(b')}d_{\psi\mu'}\psi$ be the corresponding indecomposable projective character. It follows from \autoref{decomp} that $(I^u)^{-1}(\Phi_{\mu'})$ is an integral linear combination of $\Phi_\phi$ for $\phi\in\IBr(b)$. Hence, 
\[[I^u(\phi),\Phi_{\mu'}]^0=[I^u(\phi),\Phi_{\mu'}]=[\phi,(I^u)^{-1}(\Phi_{\mu'})]=[\phi,(I^u)^{-1}(\Phi_{\mu'})]^0\in\ZZ\] for $\phi\in\IBr(b)$. 
This shows $I^u(\phi)=\sum_{\mu'\in\IBr(b')}a_{\phi\mu'}\mu'\in\ZZ\IBr(b')$. From \eqref{stuff} we obtain \[d_{I(\chi)\mu'}^{u'}=\sum_{\phi\in\IBr(b)}a_{\phi\mu'}d^u_{\chi\phi}.\] 
Now the claim follows easily (cf. proof of \autoref{decomp}).
\end{proof}

Let us compare isotypies with Brauer's notion~\cite{Brauertypes} of the \emph{type} of a block: $B$ and $B'$ are of the same type if for every $(u,b)\in S(B)$ we have
\begin{itemize}
\item $Q_{(u,b)}S_{(u,b)}=TQ_{(u',b')}$ where $T$ is a permutation matrix not depending on $(u,b)$ and $S_{(u,b)}\in\GL(l(b),\ZZ)$,

\item $Q^bS_{(u,b)}=T^uQ^{b'}$ where $Q^b$ (respectively $Q^{b'}$) is the (ordinary) decomposition matrix of $b$ (respectively $b'$) and $T^u$ is a permutation matrix.
\end{itemize}
The main difference is the absence of signs (compared to \autoref{isodec}). 

\section{Broué's conjecture}

The following conjecture is probably the main motivation to study perfect isometries.

\begin{Con}[{Broué~\cite[Conjecture~6.1]{Broue}}]\label{brouec}
If $B$ has abelian defect group $D$, then $B$ is isotypic to its Brauer correspondent $b_D$ in $\N_G(D)$.
\end{Con}

In the situation of \autoref{brouec}, the blocks $B$ and $b_D$ have the same defect group and the same fusion system. In fact the fusion system is controlled by the inertial quotient $T(B):=\N_G(D,\beta)/\C_G(D)$ where $\beta$ is a Brauer correspondent of $B$ (and of $b_D$) in $\C_G(D)$. Since $T(B)$ is a $p'$-group, 
\[D=[T(B),D]\times\C_{T(B)}(D).\] 
Let $\mathcal{R}$ be a set of representatives for the $T(B)$-orbits on $[T(B),D]$. 
Then the set $S(B)$ in \autoref{isotypy} can be defined by $\{(uv,b_{uv}):u\in\mathcal{R},\,v\in\C_{T(B)}(D)\}$ where $b_{uv}=\beta^{\C_G(uv)}$ has defect group $D$ and $T(b_{uv})\cong\C_{T(B)}(u)$. Moreover, $S(b_D)$ is given by $\{(uv,b_{uv}')\}$ where $b_{uv}'=\beta^{\C_G(uv)\cap\N_G(D)}$ is the Brauer correspondent of $b_{uv}$ in $\C_G(uv)\cap\N_G(D)$. This makes it possible to work by induction on $|G|$.

\autoref{brouec} holds at least in the following cases:
\begin{itemize}
\item $G$ $p$-solvable (see \cite[Théor{\`e}me~5.5]{Broue})
\item $G=S_n$, $A_n$, $2.S_n$ or $2.A_n$ (see \cite[Théor{\`e}me~2.13]{RouquierPerfect} and \cite{FongHarrisAn,Livesey})
\item $G$ general linear, unitary or some symplectic group (see \cite{EnguehardGL,Riss,Livesey2})
\item $G$ (almost) (quasi)simple sporadic (see \cite{SambaleBroue})
\item $G$ quasisimple with exceptional Schur multiplier (see \cite{SambaleExc})
\item $B$ nilpotent (see \cite[Théor{\`e}me~5.2]{Broue})
\item $B$ principal and $p=2$ (see \cite{FongHarris})
\item $B$ principal and $D\cong C_3\times C_3$ (see \cite{KK})
\item $|T(B)|\le 4$ or $T(B)\cong S_3$ (see \cite{UsamiZ2Z2,UsamiZ4,Usami23I,UsamiD6})
\item $[T(B),D]$ cyclic (see \cite[Théor{\`e}me~5.3]{Broue} and \cite[Corollary]{WatanabeCycFoc})
\item $p=2$ and $D$ of rank $\le3$ (see \cite[Theorem~15]{SambaleC4})
\item $|D|=16$ (see \cite{EatonE16}) and most cases for $|D|=32$ (see \cite[Proposition~5.5]{Ardito})
\end{itemize}

The case $D\cong C_3\times C_3$ is still open (cf. \autoref{3x3})!

\begin{Ex}\hfill
\begin{enumerate}[(i)]
\item In the situation of \autoref{brouec} there is not always a perfect isometry with positive signs only: Let $B$ be the principal $2$-block of $G=A_5$. Then $D=V_4$ and $\N_G(D)=A_4$. Consequently, the character degrees of $B$ are $1,3,3,5$ and those of $b_D$ are $1,1,1,3$. Hence, for any bijection $I:\Irr(B)\to\Irr(b_D)$ we have 
\[\mu_I(1,1)=\sum_{\chi\in\Irr(B)}\chi(1)I(\chi)(1)\equiv 2\pmod{4}.\]
Nevertheless, $B$ and $b_D$ are perfectly isometric via $\chi_1\mapsto-\psi_1$ and $\chi_i\mapsto\psi_i$ for $i=2,3,4$ as can be seen from the character tables (use (int')):
\begin{align*}
\begin{array}{c|ccccc}
B&1&(12)(34)&(123)&(12345)&(13524)\\\hline
\chi_1&1&1&1&1&1\\
\chi_2&3&-1&.&\frac{1-\sqrt{5}}{2}&\frac{1+\sqrt{5}}{2}\\
\chi_3&3&-1&.&\frac{1+\sqrt{5}}{2}&\frac{1-\sqrt{5}}{2}\\
\chi_4&5&1&-1&.&.
\end{array}
&&
\begin{array}{c|cccc}
b_D&1&(12)(34)&(123)&(132)\\\hline
\psi_1&1&1&1&1\\
\psi_2&1&1&\frac{-1+\sqrt{3}i}{2}&\frac{-1-\sqrt{3}i}{2}\\
\psi_3&1&1&\frac{-1-\sqrt{3}i}{2}&\frac{-1+\sqrt{3}i}{2}\\
\psi_4&3&-1&.&.
\end{array}
\end{align*}
For a non-trivial $B$-subsection $(u,b)$ the block $b$ is nilpotent. The same holds for the $b_D$-subsections. Hence, one can show that $B$ and $b_D$ are isotypic.

\item In general, \autoref{brouec} does not hold for non-abelian $D$: If $B$ is the principal $2$-block of $G=S_4$, then $D\cong D_8$ and $l(B)=2$. On the other hand, the principal $2$-block $b_D$ of $\N_G(D)=D$ satisfies $l(b_D)=1$. Hence, $B$ and $b_D$ are not perfectly isometric according to \autoref{decomp}. Nevertheless, Rouquier~\cite[A.2]{RouquierStable} put forward a variant of Broué's conjecture for blocks with abelian hyperfocal subgroup. 
This has been settled for $p=2$ and $D$ metacyclic by Cabanes--Picaronny~\cite{Cabanesrev} (in combination with \cite[Theorem~8.1]{habil}). The case $p>2$ and $D$ non-abelian, metacyclic has been done more recently by Tasaka--Watanabe~\cite{TasakaWatanabe}. Similar case were considered in \cite{Holloway,WatanabePerfIso,HuZhou,Sambalemna3,Tasakahyper,Todea,WatanabeCycFoc,WatanabeCenHyp,WatanabeAWC}.
\end{enumerate}
\end{Ex}

\begin{Prop}\label{IN}
Suppose that \autoref{brouec} holds for $B$. Then there exists $\gamma\in\ZZ$ such that
\[|\{\chi\in\Irr(B):\chi(1)_{p'}\equiv\pm\gamma k\pmod{p}\}|=|\{\psi\in\Irr(b_D):\psi(1)_{p'}\equiv \pm k\pmod{p}\}|\]
for every $k\in\ZZ$.
\end{Prop}
\begin{proof}
Let $I:\ZZ\Irr(B)\to\ZZ\Irr(b_D)$ be a perfect isometry and $\chi\in\Irr(B)$. Let $\gamma\in\ZZ$ such that $I(\chi)(1)_{p'}\gamma\equiv\chi(1)_{p'}\pmod{p}$. Now the claim follows from \autoref{degcorr}.
\end{proof}

Isaacs--Navarro~\cite{IsaacsNavarro} conjectured that $\gamma=|G:\N_G(D)|_{p'}$ works in the situation of \autoref{IN}. In particular, if $B$ has maximal defect, $\gamma=1$ by Sylow's theorem. Broué~\cite[Remarque on p. 65]{Broue} states without proof that this holds for the principal block (he informed the author that his claim relies on the additional assumption that the trivial character maps to the trivial character). 

\addcontentsline{toc}{section}{Acknowledgment}
\section*{Acknowledgment}
I like to thank Michel Broué for answering some questions during a conference in Banff. 
Moreover, I appreciate values comments by anonymous referees. 
Parts of this work were written while I was in residence at the Mathematical Sciences Research Institute in Berkeley (Spring 2018) with the kind support by the National Science Foundation (grant DMS-1440140).
This work is also supported by the German Research Foundation (projects SA \mbox{2864/1-1} and SA \mbox{2864/3-1}).

\end{document}